\documentclass[oneside]{amsart}
\usepackage{amssymb}

\makeatletter

\DeclareRobustCommand{\lyxmathsym}[1]{\ifmmode\begingroup\def\b@ld{bold}
  \def\rmorbf##1{\ifx\math@version\b@ld\textbf{##1}\else\textrm{##1}\fi}
  \mathchoice{\hbox{\rmorbf{#1}}}{\hbox{\rmorbf{#1}}}
  {\hbox{\smaller[2]\rmorbf{#1}}}{\hbox{\smaller[3]\rmorbf{#1}}}
  \endgroup\else#1\fi}

\providecommand{\tabularnewline}{\\}

\numberwithin{equation}{section} 
\numberwithin{figure}{section} 
\theoremstyle{plain}
\theoremstyle{plain}
\newtheorem{thm}{Theorem}

\makeatother

\begin{document}

\title{Cari\~{n}ena orthogonal polynomials are Jacobi polynomials}

\author{C. Vignat and P.W. Lamberti}
\address{I.G.M., Universit\'{e} de Marne la Vall\'{e}e and Facultad de Matematica, Astronomia y Fisica, Universidad Nacional de Cordoba and CONICET }
\maketitle

\section{Introduction}

The relativistic Hermite polynomials (RHP) were introduced in 1991
by Aldaya et al. \cite{Aldaya} in a generalization of the theory
of the quantum harmonic oscillator to the relativistic context. These
polynomials were later related to the more classical Gegenbauer (or
more generally Jacobi) polynomials in a study by Nagel \cite{Nagel}.
For this reason, they do not deserve any special study since their
properties can be deduced from the properties of the well-known Jacobi
polynomials - a more general class of polynomials that includes Gegenbauer
polynomials - as underlined by Ismail in \cite{Ismail}. Recently,
Cari\~{n}ena et al. \cite{Carinena} studied an extension of the quantum
harmonic oscillator on the sphere $\mathcal{S}^{2}$ and on the hyperbolic
plane and they showed that the Schrödinger equation can be analytically
solved; the solutions are an extension of the classical wavefunctions
of the quantum harmonic oscillator where the usual Hermite polynomials
are replaced by some new polynomials - that we will call here Cari\~{n}ena
polynomials. In the next section, we show that these Cari\~{n}ena polynomials
are in fact Jacobi polynomials. The last section is devoted to the
application of these results to the nonextensive context.

\title{}

\section{Relativistic Hermite and Cari\~{n}ena polynomials: definitions and notations}

\subsection{Relativistic Hermite polynomials}

\begin{flushleft}
The relativistic Hermite polynomial $H_{n}^{N}$ of degree $n$ and
parameter $N$ is defined by the Rodrigues formula\begin{equation}
H_{n}^{N}\left(X\right)=\left(-1\right)^{n}\left(1+\frac{X^{2}}{N}\right)^{N+n}\frac{d^{n}}{dX^{n}}\left(1+\frac{X^{2}}{N}\right)^{-N};\label{eq:RodriguesRHP}\end{equation}
examples of RHP polynomials are\begin{eqnarray*}
H_{0}^{N}\left(X\right)=1; & H_{1}^{N}\left(X\right)=2X; & H_{2}^{N}\left(X\right)=2\left(-1+X^{2}\left(2+\frac{1}{N}\right)\right);\\
 & H_{3}^{N}\left(X\right)=4\left(1+\frac{1}{N}\right)\left(X^{3}\left(2+\frac{1}{N}\right)-3X\right)\end{eqnarray*}
These polynomials are extensions of the classical Hermite polynomials
$H_{n}\left(X\right)$ that are defined as
\begin{equation}
\label{Hermitepolynomials}
H_{n}\left(X\right)=\left(-1\right)^{n}\exp\left(X^{2}\right)\frac{d^{n}}{dX^{n}}\exp\left(-X^{2}\right)
\end{equation}
and thus can be obtained as the limit case\[
\lim_{N\to+\infty}H_{n}^{N}\left(X\right)=H_{n}\left(X\right).\]
The RHP are orthogonal on the real line in the following sense\begin{equation}
\int_{-\infty}^{+\infty}H_{n}^{N}\left(X\right)H_{m}^{N}\left(X\right)\left(1+\frac{X^{2}}{N}\right)^{-N-1-\frac{m+n}{2}}dX=\frac{\sqrt{N\pi}n!\Gamma\left(2N+n\right)\Gamma\left(N+\frac{1}{2}\right)}{\left(n+N\right)N^{n}\Gamma\left(2N\right)\Gamma\left(N\right)}\delta_{m,n}.\label{eq:orthogonalRHP}\end{equation}
We note that this is an unconventional orthogonality since $H_{n}^{N}\left(X\right)$
and $H_{m}^{N}\left(X\right)$ are orthogonal with respect to a measure
which depends on the degrees of these polynomials.
\par\end{flushleft}

\subsection{Gegenbauer polynomials}

\begin{flushleft}
The Gegenbauer polynomial $C_{n}^{\nu}$ of degree $n$ and parameter
$\nu\ne0$ is defined by the Rodrigues formula\[
C_{n}^{\nu}\left(X\right)=\alpha_{n,\nu}\left(-1\right)^{n}\left(1-X^{2}\right)^{\frac{1}{2}-\nu}\frac{d^{n}}{dX^{n}}\left(1-X^{2}\right)^{n+\nu-\frac{1}{2}}\]
with \[
\alpha_{n,\nu}=\frac{\left(2\nu\right)_{n}}{2^{n}n!\left(\nu+\frac{1}{2}\right)_{n}};\]
examples of Gegenbauer polynomials are\begin{eqnarray*}
C_{0}^{\nu}\left(X\right)=1; & C_{1}^{\nu}\left(X\right)=2\nu X; & C_{2}^{\nu}\left(X\right)=2\nu\left(\nu+1\right)X^{2}-\nu;\\
 & C_{3}^{\nu}\left(X\right)=2\nu\left(\nu+1\right)\left(\frac{2\left(\nu+2\right)}{3}X^{3}-X\right).\end{eqnarray*}
\[
\lyxmathsym{      }\]
These polynomials are orthogonal with respect to the measure $\left(1-X^{2}\right)^{\nu-\frac{1}{2}}dX$
on $\left[-1,+1\right]:$ for $\nu\ne0$\[
\int_{-1}^{+1}C_{n}^{\nu}\left(X\right)C_{m}^{\nu}\left(X\right)\left(1-X^{2}\right)^{\nu-\frac{1}{2}}dX=\frac{\pi2^{1-2\nu}\Gamma\left(n+2\nu\right)}{n!\left(n+\nu\right)\Gamma\left(\nu\right)}\delta_{m,n}\]

\par\end{flushleft}

\subsection{Cari\~{n}ena polynomials}

\begin{flushleft}
The Cari\~{n}ena polynomial of degree $n$ and parameter $\mathcal{N\in\mathbb{R}}$
is defined by the Rodrigues formula\[
\mathcal{H}_{n}^{\mathcal{N}}\left(X\right)=\left(-1\right)^{n}\left(1+\frac{X^{2}}{\mathcal{N}}\right)^{\mathcal{N}+\frac{1}{2}}\frac{d^{n}}{dX^{n}}\left(1+\frac{X^{2}}{\mathcal{N}}\right)^{n-\mathcal{N}-\frac{1}{2}}\]
where $X\in\mathbb{R}$ for $\mathcal{N}>0$ and $X\in\left[-\sqrt{-\mathcal{N}},+\sqrt{-\mathcal{N}}\right]$
when $\mathcal{N}<0.$ Since the cases $\mathcal{N}>0$ and $\mathcal{N}<0$
differ greatly, we'll denote - for reasons that will appear clearly
in the following - the Cari\~{n}ena polynomials with negative parameter
as \[
\mathcal{H}_{n}^{\mathcal{N}}\left(X\right)=\mathcal{C}_{n}^{\nu}\left(X\right),\lyxmathsym{  }\nu=-\mathcal{N}.\]
The Cari\~{n}ena polynomials are orthogonal with respect to the measure
$\left(1+\frac{X^{2}}{\mathcal{N}}\right)^{-\mathcal{N}-\frac{1}{2}}dX$
on the real line for $\mathcal{N}>0,$ and on the interval $\left[-\sqrt{\nu},+\sqrt{\nu}\right]$
with respect to the measure $\left(1-\frac{X^{2}}{\nu}\right)^{\nu-\frac{1}{2}}dX$
when $\mathcal{N}<0$: \[
\int_{\mathbb{R}}\mathcal{H}_{n}^{\mathcal{N}}\left(X\right)\mathcal{H}_{m}^{\mathcal{N}}\left(X\right)\left(1+\frac{X^{2}}{\mathcal{N}}\right)^{-\mathcal{N}-\frac{1}{2}}dX=a_{n}\delta_{m,n}\]
and\[
\int_{-\sqrt{\nu}}^{+\sqrt{\nu}}\mathcal{C}_{n}^{\nu}\left(X\right)\mathcal{C}_{m}^{\nu}\left(X\right)\left(1-\frac{X^{2}}{\nu}\right)^{\nu-\frac{1}{2}}dX=b_{n}\delta_{m,n}\]
for some constants $a_{n}$ and $b_{n}.$
\par\end{flushleft}

\section{Links between Relativistic Hermite and Cari\~{n}ena polynomials}

The following theorems show that the family of Cari\~{n}ena polynomials
is related to the set of RHP and Gegenbauer polynomials in a simple
way.
\begin{thm}
The Cari\~{n}ena polynomial $\mathcal{H}_{n}^{\mathcal{N}}\left(X\right)$
of degree $n$ and parameter $\mathcal{N}>0$ is related to the RHP
polynomial $H_{n}^{N}\left(X\right)$ of same degree $n$ and parameter
$N$ as\begin{equation}
\mathcal{H}_{n}^{\mathcal{N}}\left(X\right)=\left(\frac{N}{\mathcal{N}}\right)^{\frac{n}{2}}H_{n}^{N}\left(X\sqrt{\frac{N}{\mathcal{N}}}\right)\label{eq:RHPtoCarinenaN>0}\end{equation}
with \[
N=\mathcal{N}+1/2-n.\]
\end{thm}
\begin{proof}
Denote $N=\mathcal{N}+\frac{1}{2}-n;$ then\[
\mathcal{H}_{n}^{\mathcal{N}}\left(X\right)=\left(-1\right)^{n}\left(1+\frac{X^{2}}{\mathcal{N}}\right)^{N+n}\frac{d^{n}}{dX^{n}}\left(1+\frac{X^{2}}{\mathcal{N}}\right)^{-N}.\]
But by the Rodrigues formula (\ref{eq:RodriguesRHP}) \[
\left(-1\right)^{n}\left(1+\frac{X^{2}}{\mathcal{N}}\right)^{N+n}\frac{d^{n}}{dX^{n}}\left(1+\frac{X^{2}}{\mathcal{N}}\right)^{-N}=\left(\frac{N}{\mathcal{N}}\right)H_{n}^{N}\left(X\sqrt{\frac{N}{\mathcal{N}}}\right)\]
so that the result holds.
\end{proof}
The same kind of result is now obtained for Cari\~{n}ena polynomials with
negative parameter, where the Gegenbauer polynomials now play the
role of the RHP polynomials.
\begin{thm}
The Cari\~{n}ena polynomial $\mathcal{C}_{n}^{\nu}\left(X\right)$ of
degree $n$ and parameter $\mathcal{N}=-\nu<0$ is related to the
Gegenbauer polynomial $C_{n}^{\nu}$ of same degree $n$ and parameter
$\nu$ as\begin{equation}
\mathcal{C}_{n}^{\nu}\left(X\right)=\frac{1}{\alpha_{n,\nu}}\nu^{-\frac{n}{2}}C_{n}^{\nu}\left(\frac{X}{\sqrt{\nu}}\right).\label{eq:GegenbauertoCarinenaN<0}\end{equation}
\end{thm}
\begin{proof}
With $\nu=-\mathcal{N},$ we deduce\[
\mathcal{C}_{n}^{\mathcal{\nu}}\left(X\right)=\left(-1\right)^{n}\left(1-\frac{X^{2}}{\nu}\right)^{\frac{1}{2}-\nu}\frac{d^{n}}{dX^{n}}\left(1-\frac{X^{2}}{\nu}\right)^{n+\nu-\frac{1}{2}}.\]
It can be easily checked that \[
\left(1-\frac{X^{2}}{\nu}\right)^{\frac{1}{2}-\nu}\frac{d^{n}}{dX^{n}}\left(1-\frac{X^{2}}{\nu}\right)^{n+\nu-\frac{1}{2}}=\frac{1}{\alpha_{n,\nu}}\left(\frac{1}{\nu}\right)^{\frac{n}{2}}C_{n}^{\nu}\left(\frac{X}{\sqrt{\nu}}\right)\]
 so that the result holds.
\end{proof}
We now use Nagel's identity \cite{Nagel}\begin{equation}
H_{n}^{N}\left(X\right)=\frac{n!}{N^{\frac{n}{2}}}\left(1+\frac{X^{2}}{N}\right)^{\frac{n}{2}}C_{n}^{N}\left(\frac{X/\sqrt{N}}{\sqrt{1+\frac{X^{2}}{N}}}\right)\label{eq:Nagel}\end{equation}
that connects the RHP polynomials $H_{n}^{N}\left(X\right)$ with
the Gegenbauer polynomials $C_{n}^{N}\left(X\right)$; we show that
the same kind of connection can be derived between Cari\~{n}ena polynomials
with positive parameter $\mathcal{H}_{n}^{\mathcal{N}}\left(X\sqrt{\mathcal{N}}\right)$
and Cari\~{n}ena polynomials with negative parameter $\mathcal{C}_{n}^{\mathcal{\nu}}\left(X\right)$
as follows.
\begin{thm}
The Cari\~{n}ena polynomial $\mathcal{H}_{n}^{\mathcal{N}}\left(X\right)$
of degree $n$ and parameter $\mathcal{N}>0$ is related to the Cari\~{n}ena
polynomial $\mathcal{C}_{n}^{\nu}\left(X\right)$ of same degree $n$
and parameter $\nu$ by the following formula\begin{equation}
\mathcal{H}_{n}^{\mathcal{N}}\left(X\sqrt{\mathcal{N}}\right)=\alpha_{n,\nu}n!\left(\frac{\nu}{\mathcal{N}}\right)^{\frac{n}{2}}\left(1+X^{2}\right)^{\frac{n}{2}}\mathcal{C}_{n}^{\nu}\left(\frac{X\sqrt{\nu}}{\sqrt{1+X^{2}}}\right)\label{eq:CarinenatoCarinena}\end{equation}
where \[
\nu=\mathcal{N}+1/2-n.\]
\end{thm}
\begin{proof}
This is a direct consequence of Nagel's identity (\ref{eq:Nagel})
and equalities (\ref{eq:RHPtoCarinenaN>0}) and (\ref{eq:GegenbauertoCarinenaN<0}).
\end{proof}
These results are summarized in Table 1. %
\begin{table}
\begin{centering}
\begin{tabular}{|ccc|}
\hline 
Relativistic Hermite $H_{n}^{N}\left(X\right)$ & $\overset{\text{(\ref{eq:RHPtoCarinenaN>0})}}{\longrightarrow}$ & Cari\~{n}ena $\mathcal{H}_{n}^{\mathcal{N}}\left(X\right)$ with $\mathcal{N}>0$\tabularnewline
\begin{tabular}{c}
\tabularnewline
Nagel's identity (\ref{eq:Nagel}) $\downarrow$\tabularnewline
\tabularnewline
\end{tabular} &  & \begin{tabular}{c}
\tabularnewline
$\downarrow$(\ref{eq:CarinenatoCarinena})\tabularnewline
\tabularnewline
\end{tabular}\tabularnewline
Gegenbauer $C_{n}^{\nu}\left(X\right)$ & $\overset{\text{(\ref{eq:GegenbauertoCarinenaN<0})}}{\longrightarrow}$ & Cari\~{n}ena $\mathcal{C}_{n}^{\nu}\left(X\right)$ with $\mathcal{N}=-\nu<0$\tabularnewline
\hline
\end{tabular}
\par\end{centering}

\medskip{}
\medskip{}

\caption{Summary of the results}

\end{table}

\section{The nonextensive setup}

In the nonextensive theory, the classical Shannon entropy of a probability
density $f_{X}$\[
H=-\int f_{X}\log f_{X}\]
is replaced by the so-called Tsallis entropy \[
H_{q}=\frac{1}{1-q}\int\left(f_{X}-f_{X}^{q}\right)\]
where $q$ is a positive real number called the nonextensivity parameter.
It can be checked by L'Hospital rule that \[
H=\lim_{q\to1}H_{q}.\]
The canonical distribution in the classical $q=1$ case - that is
the distribution with maximum entropy and given variance $\sigma^{2}$
- is known to be Gaussian distribution\begin{equation}
f_{X}\left(X\right)=\frac{1}{\sigma\sqrt{2\pi}}\exp\left(-\frac{X^{2}}{2\sigma^{2}}\right).\label{eq:Gaussian}\end{equation}

The polynomials orthogonal on the real line with respect to the Gaussian
measure are the Hermite polynomials (\ref{Hermitepolynomials}); the Hermite \textit{functions}
are defined as
\[
h_{n}\left(X\right)=\exp\left(-\frac{X^{2}}{2}\right)H_{n}\left(X\right); \,n\ge0
\]
and verify the simple orthogonality property\[
\int_{\mathbb{R}}h_{n}\left(X\right)h_{m}\left(X\right)dX=\sqrt{\pi}2^{n}n!\delta_{m,n}.\]
In the nonextensive case, the canonical distributions are called $q-$Gaussian distributions
and read, for $q<1$\[
f_{X}\left(X;q\right)=\frac{\Gamma\left(\frac{2-q}{1-q}+\frac{1}{2}\right)}{\Gamma\left(\frac{2-q}{1-q}\right)\sigma\sqrt{\pi d}}\left(1-\frac{X^{2}}{d\sigma^{2}}\right)_{+}^{\frac{1}{1-q}};\lyxmathsym{  }d=2\frac{2-q}{1-q}+1\]
and for $1<q<\frac{5}{3}$\[
f_{X}\left(X;q\right)=\frac{\Gamma\left(\frac{1}{q-1}\right)}{\Gamma\left(\frac{1}{q-1}-\frac{1}{2}\right)\sigma\sqrt{\pi\left(m-2\right)}}\left(1+\frac{X^{2}}{\left(m-2\right)\sigma^{2}}\right)^{\frac{1}{1-q}};\lyxmathsym{  }m=\frac{2}{q-1}-1\]
It can be easily checked that the limit case \[
\lim_{q\to1}f_{X}\left(X;q\right)\]
coincides with the Gaussian distribution (\ref{eq:Gaussian}). 

To our best knowledge, the polynomials orthogonal with respect to
the $q-$Gaussian distributions - the extensions of the Hermite polynomials
- have not been studied in the non-extensive theory. They can be deduced
from the results of Section 2, and are indicated in Table 2 %
\footnote{Note that we consider here the $q-$Gaussian distributions with scaling
constants normalized to 1%
}. In the case $q<1,$ the Gegenbauer polynomials are the polynomials
orthogonal with respect to the $q-$Gaussian distribution. In the
case $q>1,$ either the Cari\~{n}ena or the Relativistic Hermite polynomials
are orthogonal with respect to the $q-$Gaussian measure; in the relativistic
case, this measure should also depend on indices $m$ and $n,$ what
is not the case in the Cari\~{n}ena case. In all cases, the corresponding
orthogonal function and the domain of definition $I$ is given; each
of the corresponding orthogonal functions - let us call it generically
$w_{n}\left(X\right)$ - verifies the orthogonality property\[
\int_{I}w_{n}\left(X\right)w_{m}\left(X\right)dX=K_{n}\delta_{m,n}\]
for some constant $K_{n}.$

\begin{table}
\begin{centering}
\label{tab:table2}\begin{tabular}{|c|c|c|c|}
\cline{2-4} 
\multicolumn{1}{c|}{} & $q$ & orthogonal function & domain\tabularnewline
\hline 
\begin{tabular}{c}
Gegenbauer \tabularnewline
polynomials\tabularnewline
\end{tabular} & 
$\frac{2\nu-3}{2\nu-1}<1$
 & $\mathfrak{c}_{n}^{\nu}\left(X\right)=\left(1-X^{2}\right)^{\frac{\nu}{2}-\frac{1}{4}}C_{n}^{\nu}\left(X\right)$ & $\left[-1;1\right]$\tabularnewline
\hline 
\begin{tabular}{c}
Cari\~{n}ena\tabularnewline
polynomials\tabularnewline
\end{tabular}  & $\frac{2\mathcal{N}+3}{2\mathcal{N}+1}>1$ & $\mathfrak{h}_{n}^{\mathcal{N}}\left(X\right)=\left(1+\frac{X^{2}}{\mathcal{N}}\right)^{-\frac{\mathcal{N}}{2}-\frac{1}{4}}\mathcal{H}_{m}^{\mathcal{N}}\left(X\right)$ & $\mathbb{R}$\tabularnewline
\hline 
\begin{tabular}{c}
Relativistic Hermite\tabularnewline
polynomials\tabularnewline
\end{tabular} & $\frac{2+N+\frac{m+n}{2}}{1+N+\frac{m+n}{2}}>1$ & $h_{n}^{N}\left(X\right)=\left(1+\frac{X^{2}}{N}\right)^{-\frac{N+1+n}{2}}H_{n}^{N}\left(X\right)$ & $\mathbb{R}$\tabularnewline
\hline
\end{tabular}
\par\end{centering}

\caption{polynomials orthogonal with respect to the $q-$Gaussian measure and
the corresponding orthogonal functions}

\end{table}

It turns out that the orthogonal functions above cited describe the
behaviours of physically significant systems:
\begin{itemize}
\item as shown in \cite{Carinena}, the probability density that describes
the harmonic oscillator on a 2-dimensional surface of constant negative
curvature $\kappa$ (typically the hyperbolic plane) are
\begin{equation}
\label{fmnN}
f_{m,n,\mathcal{N}}\left(y,z\right)=\vert\mathfrak{h}_{n}^{\mathcal{N}-m-\frac{1}{2}}\left(y\right)\vert^{2}\vert\mathfrak{h}_{m}^{\mathcal{N}}\left(z\right)\vert^{2}
\end{equation}
with $z=\frac{x}{\sqrt{1+y^{2}}}$ (note that $y$ and $z$ are not
independent variables). The parameter $\mathcal{N}$ here is defined
as \[
\mathcal{N}=-\frac{m\alpha}{\hbar\kappa}>0.\]

\item as shown in the same reference, the probability density that describes
the harmonic oscillator on a 2-dimensional surface of constant positive
curvature $\kappa$ (typically the sphere) are
\begin{equation}
\label{gmnN}
g_{m,n,\nu}\left(y,z\right)=\vert\mathfrak{c}_{n}^{\nu+m+\frac{1}{2}}\left(y\right)\vert^{2}\vert\mathfrak{c}_{m}^{\nu}\left(z\right)\vert^{2}
\end{equation}
with $z=\frac{x}{\sqrt{1-y^{2}}}$ and \[
\nu=\frac{m\alpha}{\hbar\kappa}>0.\]

\item the harmonic oscillator in the relativistic context as described by
\cite{Aldaya} has probability density \[
f_{n,N}\left(X\right)=\vert h_{n}^{N}\left(X\right)\vert^{2}\]

\end{itemize}
where the parameter $N>0$ is defined as\[
N=\frac{mc^{2}}{\hbar\omega}\]
so that the non-relativistic limit $c\to+\infty$ corresponds to the
classical $N\to+\infty$ Hermite polynomials.

Thus the behaviour of the harmonic oscillator - in either the relativistic
case or the case of constant curvature geometries - can be related
to the nonextensive framework, giving in each case an explicit physical
interpretation of the nonextensivity parameter $q$ in terms of the
physical constants of the harmonic oscillator, as shown in Table 3.
We note that in the case of the harmonic oscillator on the sphere
or on the hyperbolic plane, the densities (\ref{fmnN}) and (\ref{gmnN}) are seperable functions in the  variables $z$
and $y,$ each term inducing a different value of $q.$

\begin{table}[h]
\begin{centering}
\begin{tabular}{|c|c|c|}
\hline 
positive curve $\kappa$ & $\frac{2\nu-3}{2\nu-1}$ & $\frac{\nu+m-1}{\nu+m}$\tabularnewline
\hline 
negative curve $\kappa$ & $\frac{2\mathcal{N}+3}{2\mathcal{N}+1}$ & $\frac{\mathcal{N}-m+1}{\mathcal{N}-m}$
\tabularnewline
\hline 
RHP & \multicolumn{1}{c}{$1+\frac{1}{1+\frac{mc^{2}}{\hbar N}+\frac{m+n}{2}}$} & \tabularnewline
\hline
\end{tabular}
\par\end{centering}

\caption{values of the nonextensivity parameter $q$ associated with the three
harmonic oscillators}

\end{table}

\section{A natural bijection}

\subsection{the relativistic harmonic oscillator}

Let us denote \[
f_{n,N}\left(X\right)=\vert h_{n}^{N}\left(X\right)\vert^{2},\lyxmathsym{   }g_{n,\nu}\left(Y\right)=\vert\mathfrak{c}_{n}^{\nu}\left(Y\right)\vert^{2}\]
the probability densities associated to the orthogonal functions studed
above. There exists a geometric interpretation of Nagel's formula
in the framework of non-extensivity: the ground state distribution
$f_{0,N}\left(X\right)=\vert h_{0}^{N}\left(X\right)\vert^{2}$ of
a $q-$Gaussian system $X$ with $q>1$ coincides with the ground
state $\vert\mathfrak{c}_{0}^{\nu}\left(Y\right)\vert^{2}$ of a $q-$Gaussian
system $Y$ provided that\[
Y=\frac{X/\sqrt{N}}{\sqrt{1+\frac{X^{2}}{N}}}\lyxmathsym{ }\text{and}\lyxmathsym{ }\nu=N\]
 Our main result is that this geometric interpretation holds not only
for the ground state, but for all states of the relativistic harmonic
oscillator as follows
\begin{thm}
If $X\sim f_{n,N}$ then the random variable \begin{equation}
Y=\frac{X/\sqrt{N}}{\sqrt{1+\frac{X^{2}}{N}}}\label{eq:XY}\end{equation}
is distributed according to $g_{n,\nu}$ with $\nu=N.$\end{thm}
\begin{proof}
The distribution of $Y$ defined by (\ref{eq:XY}) is\[
f_{Y}\left(Y\right)=\left(1+\frac{X^{2}}{N}\right)^{\frac{3}{2}}f_{n,N}\left(\frac{\sqrt{N}Y}{\sqrt{1-Y^{2}}}\right)\]
or equivalently\begin{eqnarray*}
f_{Y}\left(Y\right) & = & \left(1-Y^{2}\right)^{-\frac{3}{2}}\left(1-Y^{2}\right)^{n+N+1}\vert H_{n}^{N}\left(\frac{\sqrt{N}Y}{\sqrt{1-Y^{2}}}\right)\vert^{2}\\
 & = & \left(1-Y^{2}\right)^{n+N-\frac{1}{2}}\left(1+\frac{Y^{2}}{1-Y^{2}}\right)^{n}\vert C_{n}^{N}\left(Y\right)\vert^{2}\\
 & = & \left(1-Y^{2}\right)^{N-\frac{1}{2}}\vert C_{n}^{N}\left(Y\right)\vert^{2}=g_{n,\nu}\left(Y\right).\end{eqnarray*}

\end{proof}

\subsection{The harmonic oscillator on the sphere and on the hyperbolic plane}

We now extend the preceding result to the case of the harmonic oscillator
on spaces of constant curvature. We set to one all scaling constants
for simplicity.

\begin{thm}
Consider the harmonic oscillator on the hyperbolic plane described
by its coordinates $\left(x,y\right)$ and with distribution 
\[
f_{m,n,N}\left(z,y\right)=\vert\mathfrak{h}_{n}^{N-m-\frac{1}{2}}\left(y\right)\vert^{2}\vert\mathfrak{h}_{m}^{N}\left(z\right)\vert^{2}
\]
(with
$z=\frac{x}{\sqrt{1+y^{2}}}$).
If this system is transformed as \begin{equation}
X=\frac{x}{\sqrt{1+x^{2}+y^{2}}};\,\,Y=\frac{y}{\sqrt{1+x^{2}+y^{2}}}\,\Leftrightarrow\,x=\frac{X}{\sqrt{1-X^{2}-Y^{2}}};\,y=\frac{Y}{\sqrt{1-X^{2}-Y^{2}}}\label{eq:change}\end{equation}
then the new system $\left(X,Y\right)$ follows the distribution \[
g_{m,n,\nu}\left(Z,Y\right)=\vert\mathfrak{c}_{m}^{\nu+n+\frac{1}{2}}\left(Y\right)\vert^{2}\vert\mathfrak{c}_{n}^{\nu}\left(Z\right)\vert^{2}\]

where\[
\nu=N-m-n.\]
\end{thm}
\begin{proof}
As a function of $\left(x,y\right),$ the density of the harmonic
oscillator writes, in terms of Gegenbauer polynomials,\[
f_{m,n,N}\left(x,y\right)=\left(1+y^{2}\right)^{n}\left(1+x^{2}+y^{2}\right)^{-N+m-\frac{1}{2}}\vert C_{n}^{N-m-n}\left(\frac{y}{\sqrt{1+y^{2}}}\right)\vert^{2}\vert C_{m}^{N-m+\frac{1}{2}}\left(\frac{x}{\sqrt{1+x^{2}+y^{2}}}\right)\vert^{2}\]

We now perform the change of variable
(\ref{eq:change}); the distribution of the new system is obtained
as \[
\tilde{f}_{m,n,N}\left(X,Y\right)d\mu\left(X,Y\right)=f_{m,n,N}\left(x,y\right)d\mu\left(x,y\right)\]
with the measure \[
d\mu\left(x,y\right)=\frac{dxdy}{\sqrt{1+x^{2}+y^{2}}};\lyxmathsym{ }d\mu\left(X,Y\right)=\frac{dXdY}{\sqrt{1-X^{2}-Y^{2}}}\]
and since the Jacobian of the transformation $(x,y) \mapsto (X,Y)$ is\[
J=\det\left[\begin{array}{cc}
\frac{1+y^{2}}{\left(1+x^{2}+y^{2}\right)^{\frac{3}{2}}} & \frac{-xy}{\left(1+x^{2}+y^{2}\right)^{\frac{3}{2}}}\\
\frac{-xy}{\left(1+x^{2}+y^{2}\right)^{\frac{3}{2}}} & \frac{1+x^{2}}{\left(1+x^{2}+y^{2}\right)^{\frac{3}{2}}}\end{array}\right]=\left(1+x^{2}+y^{2}\right)^{-2}=\left(1-X^{2}-Y^{2}\right)^{2}\]
we deduce\begin{eqnarray*}
\tilde{f}_{m,n,N}\left(X,Y\right) & = & \left(1-X^{2}-Y^{2}\right)^{-1}\left(1-X^{2}-Y^{2}\right)^{N-m+\frac{1}{2}}\left(1+\frac{Y^{2}}{1-X^{2}-Y^{2}}\right)^{n}\\
 & \times & \vert C_{n}^{N-m-n}\left(\frac{Y}{\sqrt{1-X^{2}}}\right)\vert^{2}\vert C_{m}^{N-m+\frac{1}{2}}\left(X\right)\vert^{2}\\
 & = & \left(1-X^{2}-Y^{2}\right)^{N-n-m-\frac{1}{2}}\left(1-X^{2}\right)^{n}\vert C_{n}^{N-m-n}\left(\frac{Y}{\sqrt{1-X^{2}}}\right)\vert^{2}\vert C_{m}^{N-m+\frac{1}{2}}\left(X\right)\vert^{2}.\end{eqnarray*}
But since the distribution of the harmonic oscillator on the sphere
reads, in terms of Gegenbauer polynomials,\begin{equation}
g_{m,n,\nu}\left(X,Y\right)=\left(1-Y^{2}\right)^{m-\frac{1}{2}}\left(1-X^{2}-Y^{2}\right)^{\nu}\vert C_{n}^{\nu+m+\frac{1}{2}}\left(Y\right)\vert^{2}\vert C_{m}^{\nu}\left(\frac{X}{\sqrt{1-Y^{2}}}\right)\vert^{2},\label{eq:gmnN}\end{equation}
we deduce that\[
\tilde{f}_{m,n,N}\left(X,Y\right)=g_{n,m,\nu}\left(Y,X\right)\]
with \[
\nu=N-m-n.\]
\end{proof}

\section{Conclusion}
We have shown that the Cari\~{n}ena orthogonal polynomials are Jacobi polynomials; moreover, there exists a natural bijection between the negative and the positive curvature cases. These results hold only in the two dimensional case.

\end{document}